\documentclass[11pt,a4paper,reqno]{amsart}
\usepackage{amsmath,amssymb,amsfonts,epsfig,mathrsfs,cite, hyperref}
\usepackage[T1]{fontenc}
\usepackage{color}
\usepackage{array}
\usepackage{amsthm}
\usepackage{amstext}
\usepackage{graphicx}
\usepackage{setspace}

\usepackage[title]{appendix}

\makeatletter
\@namedef{subjclassname@2020}{%
  \textup{2020} Mathematics Subject Classification}
\makeatother

\usepackage[margin=2.5cm]{geometry}
\usepackage{color}
\usepackage{enumitem}
\usepackage{amscd,psfrag}
\usepackage{yhmath}
\usepackage[mathscr]{eucal}

\usepackage{comment}

\allowdisplaybreaks[4]

\usepackage{slashed}

\makeatletter
\pdfpageheight\paperheight
\pdfpagewidth\paperwidth

\usepackage{epstopdf}
\usepackage{chngcntr}
\counterwithin{figure}{section}
\usepackage{mathrsfs}

\usepackage{indentfirst}	

\usepackage[normalem]{ulem}
\theoremstyle{plain}

\newtheorem{definition}{Definition}
\newtheorem{theorem}[definition]{Theorem}
\newtheorem*{theorem*}{Theorem}

\newtheorem{remark}[definition]{Remark}

\newtheorem*{remark*}{Remark}
\newtheorem*{sideremark*}{Side Remark}

\newtheorem*{claim*}{Claim}
\newtheorem*{q*}{Question}
\newtheorem{lemma}[definition]{Lemma}

\newtheorem*{corollary*}{Corollary}

\newcommand{\R}{\mathbb{R}}

\newcommand{\na}{\nabla}
\newcommand{\emb}{\hookrightarrow}

\newcommand{\p}{\partial}
\newcommand{\loc}{{\rm loc}}
\newcommand{\weak}{\rightharpoonup}
\newcommand{\e}{\varepsilon}

\newcommand{\G}{\Gamma}
\newcommand{\linf}{{L^\infty}}

\newcommand{\two}{{\rm II}}
\newcommand{\M}{{\mathcal{M}}}
\newcommand{\bra}{\left\langle}
\newcommand{\ket}{\right\rangle}

\newcommand{\rnk}{{\R^{n+k}}}

\newcommand{\nap}{{\na^\perp}}

\newcommand{\sol}{\mathfrak{S}}
\newcommand{\so}{{\mathfrak{so}}}
\newcommand{\sonk}{\so{(n+k)}}

\newcommand{\ball}{{\bf B}}

\def\XXint#1#2#3{{\setbox0=\hbox{$#1{#2#3}{\int}$ }
\vcenter{\hbox{$#2#3$ }}\kern-.6\wd0}}

\newcommand{\X}{{\mathcal{X}}}

\title{Some recent developments on isometric immersions via compensated compactness and gauge transforms}

\author{Siran Li}

\address{Siran Li: School of Mathematical Sciences $\&$ CMA-Shanghai, Shanghai Jiao Tong University, No.~6 Science Buildings,
800 Dongchuan Road, Minhang District, Shanghai, China (200240)}

\email{\texttt{siran.li@sjtu.edu.cn}}


\date{\today}

\begin{document}

\maketitle
\begin{center}
\emph{Dedicated to Prof.~Gui-Qiang G. Chen on the occasion of his $60^{\text{th}}$ birthday,\\ with gratitude and admiration}
\end{center}

\begin{abstract}
We survey recent developments on the analysis of Gauss--Codazzi--Ricci equations, the first-order PDE system arising from the classical problem of isometric immersions in differential geometry, especially in the regime of low Sobolev regularity. Such equations are not purely elliptic, parabolic, or hyperbolic in general, hence calling for analytical tools for PDEs of mixed types. We discuss various recent contributions --- in line with the pioneering works by G.-Q. Chen, M. Slemrod, and D. Wang [\emph{Proc. Amer. Math. Soc.} (2010); \emph{Comm. Math. Phys.} (2010)] --- on the weak continuity of Gauss--Codazzi--Ricci equations, the weak stability of isometric immersions, and the fundamental theorem of submanifold theory with low regularity. Two mixed-type PDE techniques are emphasised throughout these developments: the method of compensated compactness and the theory of Coulomb--Uhlenbeck gauges. 
\end{abstract}

\section{Introduction}\label{sec: intro}
\subsection{Isometric immersions}
Isometric immersions and/or embeddings of Riemannian manifolds have long been a central  topic in global analysis and geometric PDEs. Studies on the existence, uniqueness, stability, and regularity of isometric immersions (embeddings) of Riemannian manifolds into Euclidean spaces or general ambient manifolds abound in the literature. These works have partially shaped the landscapes of analysis, geometry, and PDE nowadays, and have found wide-ranged, in-depth applications in mathematical relatively, continuum mechanics, and mathematical biology, etc. Let us  mention here the groundbreaking works of historical significance by E.~Cartan \cite{cartan}, Weyl \cite{weyl}, Aleksandroff \cite{ale},  Nirenberg \cite{nir},  Nash \cite{nashC1, nashCk}, and Pogorelov \cite{pog}; this list is by no means exhaustive. We also refer the reader to Han--Hong \cite{hh} for a comprehensive survey, and the problem section in Yau  \cite{yau}  for related questions.

This contribution focuses on isometric immersions of weak regularities. Consider a compact Riemannian manifold $(\M,g)$ of dimension $n \geq 2$. A mapping $\Phi: \M \to \rnk$ is an immersion if $d \Phi: T\M \to T\rnk$ is injective and an isometry if $\Phi^\# \delta =g$, where $\delta$ denotes the Euclidean metric on $\rnk$ and $\Phi^\#$ the pullback under $\Phi$. By ``weak regularity'' we usually mean $\Phi \in W^{2,p}$ and $g \in W^{1,p} \cap L^\infty$ with some $p \in [1,\infty]$. We emphasise the PDE approach throughout: all the results discussed in this survey are obtained by analysing weak (distributional) solutions to the Gauss--Codazzi--Ricci equations and their variants associated to the isometric immersions. 

\subsection{The PDE system}
Gauss--Codazzi--Ricci equations are the compatibility equations of curvatures for the existence of an isometric immersion $\Phi:(\M^n,g)\to\rnk$. In geometrical terms, the \emph{intrinsic geometry} of $\Phi$ is given by the metric $g$, while the \emph{extrinsic geometry} is given by the second fundamental form $\two$ 
and the normal connection $\nap$ (\emph{i.e.}, the orthogonal projection of the Levi-Civita connection on $\rnk$ to the normal bundle of $\Phi$, which is trivial when the codimension $k=1$). The curvature components in the tangential and normal directions of $\Phi$ together constitute the flat Riemann curvature on $\rnk$. This gives rise to the  \emph{Gauss--Codazzi--Ricci} equations associated to $\Phi:(\M^n,g)\to\rnk$ \cite{doc}:
\begin{eqnarray}
&& \delta\big(\two(X,Z), \two (Y,W)\big) - \delta\big( \two(X,W),\two(Y,Z) \big) = R(X,Y,Z,W),\label{gauss}\\
&& \overline{\na}_Y\two(X,Z) - \overline{\na}_X\two(Y,Z)=0,\label{codazzi}\\
&& g\big([\mathcal{S}_\eta, \mathcal{S}_{\zeta}] X, Y\big) = R^E(X,Y,\eta,\zeta)\label{ricci},
\end{eqnarray}
for any $X,Y,Z,W \in \G(T\M)$ and $\eta,\zeta \in \G(E)$. Here $[\bullet,\bullet]$ is the commutator, $\delta$ the Euclidean metric, $\mathcal{S}$ the shape operator (equivalent to $\two$ modulo contractions by $g$), $R$ and $R^E$ the Riemann curvature tensors for $(T\M,g)$ and $\left(E,g^E\right)$, respectively, and $\overline{\na}$ the Levi-Civita connection on $\R^{n+k}$. For the moment we take $E = T^\perp\M := T\rnk \slash T[\Phi(\M)]$, the normal bundle of $\Phi$.

The Gauss--Codazzi--Ricci Equations~\eqref{gauss}--\eqref{ricci} are a first-order nonlinear PDE system with zeroth-order quadratic nonlinearities. Given any local co-ordinate system $\{x^1, \ldots, x^{n+k}\}$, set
\begin{equation}\label{h and kappa}
h^\alpha_{ij}:=\langle\two(\p_i,\p_j),\p_\alpha\rangle\qquad\text{and}\qquad \kappa^\alpha_{i\beta}:=\left\langle\na^\perp_{\p_i}\p_\beta,\p_\alpha\right\rangle.
\end{equation}
Here and hereafter, we adopt the index convention: 
\begin{equation*}
1\leq i,j,k,\ell,p,q \leq n; \qquad n+1 \leq \alpha,\beta,\gamma \leq n+k;\qquad 1\leq a,b,c \leq n+k,
\end{equation*}
and write $\bra\bullet,\bullet\ket$ for both $g$ and $\delta$ (as $\Phi$ is an isometry). Then, with Einstein's summation convention, the Gauss--Codazzi--Ricci Equations~\eqref{gauss}--\eqref{ricci} are expressed locally as follows:
\begin{equation}\label{gauss, loc}
g_{\alpha\beta}\left(h^\alpha_{ji}h^\beta_{k\ell}-h^\alpha_{ki}h^\beta_{j\ell}\right)=R_{ijk\ell},
\end{equation}
\begin{equation}\label{codazzi, loc}
\frac{\p h^\alpha_{j\ell}}{\p x^k} - \frac{\p h^\alpha_{kj}}{\p x^\ell} = \Gamma^m_{kj} h^\alpha_{\ell p} - \G^p_{\ell j} h^\alpha_{kp} + \kappa^\alpha_{\ell \beta} h^\beta_{kj} - \kappa^\alpha_{k\beta} h^\beta_{\ell j},
\end{equation}
and
\begin{equation}\label{ricci, loc}
\frac{\p\kappa^\alpha_{\ell \beta}}{\p x^k} - \frac{\p \kappa^\alpha_{k\beta}}{\p x^\ell} = g^{pq} \left[ h^\alpha_{p\ell} h^\beta_{kq} - h^\alpha_{pk} h^\beta_{\ell q} \right] - \kappa^\alpha_{k\gamma}\kappa^\gamma_{\ell \beta} + \kappa^\alpha_{\ell \gamma} \kappa^\gamma_{k\beta}.
\end{equation}
Without further restrictions on curvatures, this PDE system has no fixed type --- not purely elliptic, parabolic, or hyperbolic in general.

The analysis of Gauss--Codazzi(--Ricci) equations is a classical and challenging topic in PDE theory. Most of the known results in this direction are restricted to the case $(n,k)=(2,1)$ with positive Gaussian curvature $K_\M$, whence the Gauss--Codazzi equations are equivalent to a fully nonlinear elliptic Monge--Amp\`{e}re equation \cite{nir}. The cases of nonpositive or sign-changing $K_\M$ (\emph{i.e.}, when the Gauss--Codazzi equations are hyperbolic or mixed-type PDEs) and of higher dimensions and/or codimensions, despite sporadic breakthroughs \cite{chhw, ccswy}, remain elusive in the large. See \cite{hh} for an extensive account of  historical  and current developments.

\subsection{Intrinsic approach to elasticity and fundamental theorem of submanifold theory}

The study of isometric immersions of weak regularities derives motivations from and, in turn, has brought about novel insights to many fields of applied sciences. In elasticity theory, for example, one  major objective is to determine the deformation undergone by an elastic body in response to external forces and boundary conditions. The elastic body and deformation are modelled respectively by a manifold $\M^n$ and a mapping $\Phi: {\M}^n \to \rnk$. The \emph{intrinsic approach to nonlinear elasticity}, pioneered by Antman \cite{antman} and Ciarlet \cite{ciarlet, ciarlet-new}, recasts the problems concerning the deformation $\Phi$ to those concerning the Cauchy--Green tensor $g=\Phi^\#\delta$. 

The intrinsic approach to nonlinear elasticity is based on the fundamental theorem of surface theory (Ciarlet--Gratie--Mardare \cite{cgm};\emph{cf}. \cite{ciarlet', b, add, add'} too). It ascertains that, under suitable topological and regularity conditions,\footnote{More precisely, suppose for instance that the domain is simply-connected, the metric $g$ is of $W^{1,p}_\loc \cap L^\infty_\loc$-regularity, and the second fundamental form $\two$ is of $L^p_\loc$-regularity with $p > n$.} the deformation $\Phi: {\M}^n \to \rnk$ for $n=2$, $k=1$ can be recovered from the Cauchy--Green tensor $g$ and the second fundamental form $\two$. This result admits a natural extension to isometric immersions of arbitrary dimension $n$ and codimension $k$, which we term as the ``\emph{fundamental theorem of submanifold theory}'':\footnote{One subtlety lurks here: before establishing the existence of the isometric immersion, one cannot talk about the ``normal directions''. This is why we consider the Gauss--Codazzi--Ricci equations on an abstract rank-$k$ bundle $E$: it serves as the putative normal bundle here.}
\begin{quotation}
\emph{Given a weak solution $\mathfrak{S}=\left(\two,\na^\perp\right)$ to the Gauss--Codazzi--Ricci equations on a Riemannian manifold $(\M^n,g)$ and a rank-$k$ vector bundle $\pi: E\to \M$, one may find an isometric immersion $\Phi: \M \to \rnk$ whose normal bundle $T^\perp[\Phi(\M)]$ and extrinsic geometry coincide with  $E$ and $\mathfrak{S}$, respectively.} 
\end{quotation}

\subsection{Two papers by Chen--Slemrod--Wang: weak continuity and existence of isometric immersions via compensated compactness}
G.-Q. Chen, M. Slemrod, and D. Wang (2010) published two papers \cite{csw, csw'} on the stability and existence of Gauss--Codazzi--Ricci equations and isometric immersions, which are the  starting point of many subsequent developments surveyed here. With their expertise in conservation laws and compressible fluid dynamics, the authors took a fresh look at this geometric problem and exploited one of the central tools for their fields, the div-curl lemma from the theory of compensated compactness, thus obtaining:

\begin{theorem}[Chen--Slemrod--Wang \cite{csw, csw'}]\label{thm: csw}
The following holds for the Gauss--Codazzi--Ricci equations and/or isometric immersions.
\begin{enumerate}
    \item Assume that $\{\sol^\e\}$ is a sequence of approximate solutions to the Gauss--Codazzi--Ricci equations with arbitrary dimension $n \geq 2$ and codimension $k \geq 1$. If $\{\sol^\e\}$ is bounded in $L^p_\loc$ for $p > 2$, then any of its weak subsequential limits is a weak solution to the Gauss--Codazzi--Ricci equations. 
    \item  Several families of surface metrics with negatively Gaussian curvature admit \emph{global} $C^{1,1}$-isometric immersions into $\R^3$.
\end{enumerate}    
\end{theorem}

Theorem~\ref{thm: csw} (1) establishes the \emph{weak continuity} of Gauss--Codazzi--Ricci equations, which is an important topic for the studues on nonlinear PDEs and calculus of variations; see Evans \cite{evans}. By ``approximate solutions'' we mean $\sol^\e = (\two^\e, \na^{\e,\perp}) \in L^p_\loc$ satisfying the Gauss--Codazzi--Ricci equations with error terms converging to zero in  $W^{-1,p}_\loc$. Then the authors deduced Theorem~\ref{thm: csw} (2) from (1), by way of constructing approximate solutions for $(n,k)=(2,1)$ via vanishing viscosity arguments and establishing uniform estimates in $L^\infty_\loc$ via the method of invariant regions. 
The metrics obtained in (2) include a one-parameter family of surface metrics containing the standard catenoid; more examples, such as certain (not necessarily small!) perturbations of the helicoid and Enneper surface, are found in \cite{chw, me-ARMA}.

\subsection{Fundamental theorem of submanifold theory; weak stability of isometric immersions}
The above weak continuity result,  Theorem~\ref{thm: csw} (1), is a PDE theorem in nature. More work is needed to translate it into its \emph{geometrical} counterpart, \emph{i.e.}, the stability of $W^{2,p}_\loc$-isometric immersions  with respect to $L^p_\loc$-perturbations of the extrinsic geometries. The key is the aforementioned ``fundamental theorem of submanifold theory'', which allows us to construct isometric immersions from weak solutions to the Gauss--Codazzi--Ricci equations.

Indeed, with the help of the analytical lemmas on Pfaff and Poincar\'{e} systems in \cite{m05, m07} (see \S\ref{sec: Su-Li} for details), the following version of the fundamental theorem of submanifold theory (\emph{a.k.a.} the ``realisation theorem'' \cite{cl}) has been established for $W^{2,p}$-isometric immersions; $p>n$.

\begin{theorem}[S. Mardare \cite{m05, m07}; Szopos \cite{sz}]\label{thm: fundamental, Mardare-Szopos}
Let $(\M^n,g)$ be a simply-connected closed Riemannian manifold; $n \geq 2$. Suppose that $\mathfrak{S} \in L^{p}$ with $p>n$ is a weak solution to the Gauss--Codazzi--Ricci equations on $\M$ with arbitrary codimension $k$. There exists an isometric immersion $\Phi: (\M,g) \emb (\R^{n+k},\delta)$ of $W^{2,p}$-regularity whose extrinsic geometry coincides with $\mathfrak{S}$.  Moreover, $\Phi$ is unique modulo Euclidean rigid motions in $\R^{n+k}$ outside null sets. 
\end{theorem}

From here, the weak stability of $W^{2,p}_\loc$-isometric immersions can be readily deduced:

\begin{theorem}\label{thm: weak stability of isom imm, W2p}
Let $\left\{\Phi^\e\right\}_{\e>0}$ be  uniformly bounded $W^{2,p}_{\loc}$-immersions of an $n$-dimensional domain $\M \subset \R^n$ into $\R^{n+k}$, where $p> n \geq 2$. Assume that for any compact subset $\mathcal{K} \Subset \M$, there are constants $0<c \leq C<\infty$ (depending on $\mathcal{K}$) such that for almost every $x \in \mathcal{K}$, one has
\begin{align*}
c \leq \Big\{\text{eigenvalues of the matrix } g^\e(x):=d\Phi^\e(x)\otimes d\Phi^\e(x)\Big\} \leq C.
\end{align*} 
Then $\{\Phi^\e\}$ converges weakly in $W^{2,p}_{\loc}$ modulo subsequences to an immersion $\overline{\Phi}$, whose induced metric $\overline{\Phi}^\#\delta$ is a 
limiting point of $g^\e$ in $W^{1,p}_{\loc}$ and whose extrinsic geometry (\emph{i.e.}, second fundamental form and normal connection) is a limiting point of that of $\Phi^\e$ in $L^{p}_{\loc}$. 
\end{theorem}
Variants of Theorems~\ref{thm: fundamental, Mardare-Szopos} and \ref{thm: weak stability of isom imm, W2p} have been proved in \cite{m07, sz, ciarlet', new1} by S. Mardare, Szopos, Ciarlet--Larsonneur, and Ciarlet--C. Mardare, and via geometrical arguments in \cite{cl, cl2, new2} by Chen, Slemrod, and the author. 

It is crucial to note the following discrepancy between the regularity assumptions for the geometrical \emph{vs.} analytical results on weak continuity/stability:
\begin{remark}\label{remark: key for W2p}
Theorems~\ref{thm: fundamental, Mardare-Szopos} and \ref{thm: weak stability of isom imm, W2p} on the fundamental theorem of submanifold theory and the weak stability of $W^{2,p}_\loc$-isometric immersions require $p>n$, while the weak continuity Theorem~\ref{thm: csw} (1) for Gauss--Codazzi--Ricci equations requires $p>2$. 

Also, S. Mardare \cite{m05} suggested that Theorem~\ref{thm: fundamental, Mardare-Szopos} may be sharp; \emph{i.e.}, the fundamental theorem of submanifold theory might be unachievable for $W^{2,p}_\loc$-isometric immersions with $p \leq n$.
\end{remark}

As an attempt to bridge such discrepancy, X. Su and the author  \cite{ls} succeeded in weakening the assumption $\{\Phi^\e\} \subset W^{2,p}_\loc$ for $p>n$ in Theorem~\ref{thm: fundamental, Mardare-Szopos} to $\{\Phi^\e\} \subset \mathcal{X}$ ($n \geq 3$). Here $\X$ is a function space satisfying 
\begin{align*}
    W^{2,n}_\loc \subsetneq \mathcal{X} \subsetneq W^{2,2}_\loc.
\end{align*}
More precisely, $\mathcal{X} = L^{q,n-q}_{2,w,\loc}$ with $q \in ]2,n]$; \emph{i.e.}, the second derivatives $\p_{ij}f$ of any $f \in \mathcal{X}$ lie in the local weak Morrey space $L^{q,n-q}_{w,\loc}$, which interpolates between $L^n_{w,\loc}$ (the local weak $L^n$-space) and $L^2_\loc$. Such spaces first appeared in Morrey's seminal work \cite{morrey} on regularity of harmonic mappings from 2-dimensional domains into manifolds. For $1 \leq q<\infty$ and $0 \leq \lambda \leq n$, the weak Morrey space $L^{q, \lambda}_w(U)$ consists of functions $f \in L^q_w(U)$ satisfying 
$$
\|f\|_{L^{q, \lambda}(U)} := \sup _{x \in U,\, 0<r\le \operatorname{diam}(U)} r^{-\lambda / q}\|f\|_{L^q_w\left(B_r(x) \cap U\right)}<\infty.
$$
We also define inductively $f \in L_{\ell,w}^{p,\lambda}(U)$ by $\na f \in L_{\ell-1,w}^{p,\lambda}(U)$ for each $\ell = 1,2,3,\ldots$, and define the local version of function spaces as usual.

The main result in \cite{ls} (Theorem~1.3 therein) establishes the following version of the fundamental theorem of submanifold theory, which extends Theorem~\ref{thm: fundamental, Mardare-Szopos} above. 

\begin{theorem}\label{thm: fundamental}
Let $(\M^n,g)$ be a simply-connected closed Riemannian manifold; $n \geq 3$. Suppose that $\mathfrak{S} \in L^{q,n-q}_w$ with $2<q\leq n$ is a weak solution to the Gauss--Codazzi--Ricci equations on $\M$ with arbitrary codimension $k$. There exists an isometric immersion $\Phi: (\M,g) \emb (\R^{n+k},\delta)$ in the regularity class $\X = L^{q,n-q}_{2,w}$ whose extrinsic geometry coincides with $\mathfrak{S}$.  Moreover, $\iota$ is unique modulo Euclidean rigid motions in $\R^{n+k}$ outside null sets. 
\end{theorem}

\begin{remark}\label{rem: litzinger}
The fundamental theorem of surface theory (\emph{i.e.}, the case $(n,k)=(2,1)$), together with the corresponding weak stability result, has been proved for $W^{2,2}$-isometric immersions of surfaces into $\R^3$ by Litzinger \cite{litzinger}. Our proof of Theorem~\ref{thm: fundamental} is largely motivated by this work.
\end{remark}

We also deduced the weak stability result \cite[Theorem~5.1]{ls} for isometric immersions as follows, which is a natural generalisation of Theorem~\ref{thm: weak stability of isom imm, W2p}.

\begin{theorem}\label{thm: Li--Su weak stability}
Let $\left\{\Phi^\e\right\}_{\e>0}$ be uniformly bounded $L^{q,n-q}_{2,w,\loc}$-immersions of an $n$-dimensional domain $\M \subset \R^n$ into $\R^{n+k}$, where $2<q\leq n$. Assume the following conditions:
\begin{enumerate}
\item
for any compact subset $\mathcal{K} \Subset \M$, there are constants $0<c \leq C<\infty$ (depending on $\mathcal{K}$) such that for almost every $x \in \mathcal{K}$, one has
\begin{align}\label{imm, new}
c \leq \Big\{\text{eigenvalues of the matrix } g^\e(x):=d\Phi^\e(x)\otimes d\Phi^\e(x)\Big\} \leq C;
\end{align}
\item
for each $j \in \{1,2,\,\ldots,n\}$ and $\e>0$, the first derivative $\p_j \Phi^\e$ lies in $L^{q,n-q}_{1,w,\loc} \cap L^\infty_\loc$. 
\end{enumerate}
 
Then, modulo subsequences,  $\{\Phi^\e\}$ converges  weakly in $L^{q,n-q}_{2,w,\loc}$ to an immersion $\overline{\Phi}$, whose induced metric $\overline{\Phi}^\#\delta$ is a 
limiting point of $g^\e$ in $L^{q,n-q}_{1,w,\loc}$, and whose extrinsic geometry (\emph{i.e.}, second fundamental form and normal connection) is a 
limiting point of that of $\Phi^\e$ in $L^{q,n-q}_{w,\loc}$. 
\end{theorem}

The nondegeneracy condition~\eqref{imm, new} of the induced metrics $g^\e$ is indispensable: counterexamples for $W^{2,n}$-isometrically immersed hypersurfaces in $\R^{n+1}$ lacking the weak stability property have been constructed, by Langer \cite{langer} for $n=2$ and  by the author \cite{me-PAMS} for $n \geq 3$.

The key to the proof of Theorem~\ref{thm: fundamental} is to recast the Gauss--Codazzi--Ricci equations into the \emph{Cartan formalism}, and then establish the existence of local Coulomb--Uhlenbeck gauges in the corresponding weak regularity regime. The local flatness of the gauge transformed connection 1-forms, proved using both the compensated compactness arguments in \cite{csw} and the Hardy-BMO duality, yields the desired isometric immersion. Loosely speaking, putting into the Coulomb--Uhlenbeck gauge amounts to choosing a good co-ordinate frame in a global manner,\footnote{More precisely we mean a ``frame'' (a section of the frame bundle) instead of a ``co-ordinate system'' here. But we intentionally blur such distinctions, for co-ordinates are probably more natural for the PDE-oriented reader.} in which the Gauss--Codazzi--Ricci equations possess better analytical properties.

The Morrey space $L^{2,n-2}_\loc$ (noting that 
$L^{q,n-q}_{w,\loc}$ embeds continuously into it for $2 < q \leq n$) is \emph{critical} regardless of the codimension $k$ for gauge-theoretic considerations. This phenomenon is closely related to the partial regularity theory for harmonic mappings in arbitrary dimensions and codimensions. See Rivi\`{e}re--Struwe \cite{rs}. 

\subsection{Organisation} 
The remaining parts are organised as follows.

In \S\ref{sec: weak continuity} we revisit the main results in Chen--Slemrod--Wang \cite{csw, csw'} within the geometrical framework of Cartan formalism. In \S\ref{sec: Su-Li} we discuss our recent preprint \cite{ls} on the fundamental theorem of submanifold theory with lower regularity assumptions than the putative sharp threshold suggested by S. Mardare in \cite{m05}, featuring an interplay of compensated compactness methods and gauge theoretic arguments. Finally, several open problems are discussed in \S\ref{sec: concl}.

\section{Weak continuity of Gauss--Codazzi--Ricci equations via compensated compactness}\label{sec: weak continuity}

In this section, we present an overview of the main results in Chen--Slemrod--Wang \cite{csw, csw'}. Our approach is more geometrical in nature, which may depart occasionally from the original treatment in \cite{csw, csw'}, but turns out to be convenient for further developments. 

\subsection{Div-curl structure via the Cartan formalism}

The weak continuity result, Theorem~\ref{thm: csw} (1), is proved in \cite{csw} by exploiting the ``div-curl structure'' of  Gauss--Codazzi--Ricci Equations~\eqref{gauss, loc}--\eqref{ricci, loc} in local co-ordinates, and then applying the div-curl lemma of Murat and Tartar \cite{m, t1, t2}. The crucial structural information utilised in the argument is as follows:\footnote{In fact, not all of these information has been utilised here. See Remark~\ref{rem}.}
\begin{itemize}
    \item 
    the principal parts of Equations~\eqref{gauss, loc}--\eqref{ricci, loc}, \emph{i.e.}, the first-order linear differentials of the unknowns $h=\left\{h^\alpha_{ij}\right\}$ and $\kappa=\left\{\kappa^\alpha_{i\beta}\right\}$ (see Equation~\eqref{h and kappa} for definition) in the Codazzi $\&$ Ricci Equations~\eqref{codazzi, loc}, \eqref{ricci, loc}, are ``curl-type quantities'' of the form $\p_a A_b - \p_b A_a$;  and 
    \item 
    the nonlinear terms appearing on the zeroth order in Equations~\eqref{gauss, loc}--\eqref{ricci, loc} are quadratic combinations of $h$ and $\kappa$, which take the antisymmetric form $A_aA_b-A_aA_b$. 
\end{itemize}
In the above, $A_a$ are suitable 2-tensors consisted of components of $h$ and $\kappa$. Equivalently, $A_a$ are matrices, so their multiplications are not commutative.

The structural information above becomes transparent in E. Cartan's exterior caluclus formalism adapted to submanifold theory (abbreviated in the sequel as \emph{the Cartan formalism}), which is a classical topic in differential geometry \cite{spivak}. In local co-ordinates it reads:
\begin{eqnarray}
&&d\omega^i = \sum_j \omega^j \wedge \Omega^i_j;\label{first structure eq}\\
&&0=d\Omega^a_b + \sum_c \Omega^c_b\wedge \Omega^a_c,\label{second structure eq}
\end{eqnarray}
where $\{\omega^i\}_{1\leq i \leq n}$ is the orthonormal coframe on $(T^*\M,g)$ dual to $\{\p\slash\p_i\}_{i=1}^n$, and $\{\Omega^a_b\}_{1\leq a,b \leq n+k}$ is defined entry-wise by
\begin{eqnarray}
&& \Omega^i_j (\p_k) := g(\na_{\p_k}\p_i,\p_j);\label{Omega 1}\\
&& \Omega^i_\alpha (\p_j) \equiv -\Omega^\alpha_i (\p_j) := g^E\big(\two(\p_i,\p_j), \eta_\alpha\big);\label{Omega 2}\\
&& \Omega^\alpha_\beta(\p_j):=g^E\big(\na^E_{\p_j}\eta_\alpha,\eta_\beta\big).\label{Omega 3}
\end{eqnarray}
In the above, $\{\p_i\}$ is the orthonormal frame for $(T\M,g)$ dual to $\{\omega^i\}$, and $\Omega=\{\Omega^a_b\}$ is called the \emph{connection 1-form}. As in \S\ref{sec: intro}, $E$ is a rank-$k$ bundle over $\M$ equipped with a Riemannian metric $g^E$, its Levi-Civita connection $\na^\perp$, and an orthonormal frame $\{\eta_\alpha\}_{\alpha=n+1}^{n+k}$. If the existence of isometric immersion is given, then $E$ coincides with the normal bundle $T^\perp\M$; also, recalling Equation~\eqref{h and kappa} one has that $\Omega^i_\alpha(\p_j)=h^\alpha_{ij}$ and $\Omega^\alpha_\beta(\p_j)=\kappa^\beta_{i\alpha}$.

The \emph{second structural equation} of the Cartan formalism (Equation~\eqref{second structure eq}) is equivalent to the Gauss--Codazzi--Ricci Equations~\eqref{gauss}--\eqref{ricci} as purely algebraic identities, hence in the sense of distributions too.\footnote{The first structural equation~\eqref{first structure eq} is equivalent to the torsion-free condition of the Levi-Civita connection $\na$.} The connection 1-form takes values in the Lie algebra of antisymmetric matrices. Equations~\eqref{Omega 1}--\eqref{Omega 3} are written compactly as follows (${}^\top$ for transpose):
\begin{align*}
    \Omega = \begin{bmatrix}
        \na & \two\\
        -\two^\top & \na^E
    \end{bmatrix} \in \G \big(\M; T^*\M\otimes\sonk\big).
\end{align*}
We view the second structural Equation ~\eqref{second structure eq} as an identity on $\G \left(\M; \bigwedge^2 T^*\M\otimes\sonk\right)$, the space of antisymmetric matrix-valued 2-forms, and commonly denote
\begin{equation*}
    d\Omega + [\Omega \wedge \Omega] = 0.
\end{equation*}
Here $d$ acts on the differential form factor (\emph{i.e.}, $\bigwedge^\bullet T^*\M$) and $[\bullet \wedge \bullet]$ designates the intertwining of wedge product on the form factor and Lie bracket on the matrix factor (\emph{i.e.}, $\sonk$).

\subsection{Tools from the compensated compactness theory}
As the key analytical tool for our proof of  Theorem~\ref{thm: csw} (1), the wedge product theorem (Lemma~\ref{lemma: RRT} below; Robbin--Rogers--Temple \cite[Theorem~1.1]{rrt}) of the compensated compactness theory is a generalisation of the classical div-curl lemma \emph{\`{a} la} Murat--Tartar \cite{m, t1, t2}. It has the advantage of structural simplicity and symmetry: the first-order differential constraints are imposed for $d$ of both sequences $\{\alpha^\e\}$, $\{\beta^\e\}$ (in contrast to those in the classical version, \emph{i.e.}, divergence of one sequence and curl of the other), and it can be easily extended to the multilinear case.
\begin{lemma}\label{lemma: RRT}
Let $\alpha^\e \weak \overline{\alpha}$ in $L^s_\loc(\R^n)$ and $\beta^\e \weak \overline{\beta}$ in $L^r_\loc(\R^n)$,
where $\{\alpha^\e\}, \{\beta^\e\}, \overline{\alpha}$, and $\overline{\beta}$ are differential forms over $\R^n$
and $\frac{1}{s}+\frac{1}{r}=1$. Assume that $\{d\alpha^\e\} \subset W^{-1,s}_\loc(\R^n; T^*\R^n)$ and $\{d\beta^\e\} \subset W^{1,r}_\loc(\R^n; T^*\R^n)$ are pre-compact. Then $\alpha^\e \wedge \beta^\e\, \to\, \overline{\alpha} \wedge \overline{\beta}$ in the sense of distributions.
\end{lemma}

\subsection{A proof for weak continuity  of Gauss--Codazzi--Ricci}
With the above preparations, we are ready to sketch a simple proof of the weak continuity property of Gauss--Codazzi--Ricci equations. The current form of the proof stays close to that in \cite{cl2, cg}.
\begin{proof}[Proof of Theorem~\ref{thm: csw} (1)]
    Consider the connection 1-forms $\{\Omega^n\} \subset L^p_\loc\big(\M; T^*\M\otimes\sonk\big)$ associated to the sequence of approximate solutions to the Gauss--Codazzi--Ricci equations. Then
    \begin{align*}
        d\Omega_n + [\Omega_n \wedge \Omega_n] = \eta_n \,\longrightarrow\, 0 \qquad \text{ in } W^{-1,p}_\loc. 
    \end{align*}

By assumption, one has the weak convergence $\Omega^n \weak \overline{\Omega}$ in $L^p_\loc$ (and hence in $L^2_\loc$), so $\{d\Omega_n\}$ is bounded in $W^{-1,p}_\loc$.  On the other hand, $\{[\Omega_n \wedge \Omega_n]=\eta_n-d\Omega_n\}$ is bounded in $L^{p/2}_\loc$ by Cauchy--Schwarz, where $p>2$. By Sobolev embedding, Rellich Lemma, and an interpolation of negative order Sobolev spaces, we deduce that $\{d\Omega_n\}$ is precompact in $W^{-1,2}_\loc$.

Applying Lemma~\ref{lemma: RRT} to $r=s=2$ and $\alpha^\e=\beta^\e=\Omega_n$, one obtains $[\Omega_n \wedge \Omega_n] \to \left[\overline{\Omega}\wedge\overline{\Omega}\right]$ in the sense of distributions. Thus $d\overline{\Omega} + \left[\overline{\Omega} \wedge \overline{\Omega}\right]=0$; \emph{i.e.}, the weak $L^p_\loc$-limits of the extrinsic geometric quantities are weak solutions to the Gauss--Codazzi--Ricci equations.  \end{proof}


Two important remarks are in order:

\begin{remark}\label{rem}
In the proof above, only the quadratic structure of the nonlinear terms in the Gauss--Codazzi--Ricci equations has been exploited (via the Cauchy--Schwarz inequality applied to $[\Omega_n \wedge \Omega_n]$). That $\Omega_n$ is an antisymmetric matrix-valued differential form has played no role by now. 
\end{remark}

\begin{remark}\label{rem+}
Invoking the harmonic analytic version of the compensated compactness theory \emph{\`{a} la} Coifman--Lions--Meyer--Semmes \cite{clms}, one may strengthen the conclusion in Lemma~\ref{lemma: RRT} to
\begin{align*}
    \alpha^\e \wedge \beta^\e \, \longrightarrow \, \overline{\alpha} \wedge \overline{\beta}\qquad \text{in } \mathcal{H}^1_{\loc},
\end{align*}
where $\mathcal{H}^1_{\loc}$ is the local Hardy space.
\end{remark}

\subsection{Existence of isometric immersions via compensated compactness}
Now let us explain the proof for the existence of isometric immersions of certain families of negatively curved surfaces in Theorem~\ref{thm: csw} (2).  Chen--Slemrod--Wang leveraged Theorem~\ref{thm: csw} (1) as a \emph{compensated compactness framework}: once a sequence of approximate solutions $\{\sol^\e\}$ has been constructed and shown to be uniformly bounded in $L^p_\loc$ ($p>2$), we may pass to the limit to obtain a weak solution via Lemma~\ref{lemma: RRT}. Here $(n,k)=(2,1)$, \emph{i.e.}, we are considering isometric immersions of surfaces into $\R^3$, in which case there are 1 Gauss, 2 Codazzi, and 0 (trivial) Ricci equations.

In \cite{csw'}, the authors exploited the approach of vanishing artificial viscosity as follows:
\begin{enumerate}
    \item 
     eliminating one of the three components of the second fundamental form $h=\{h_{ij}\}_{1\leq i, j \leq 2}$ ($h_{12}=h_{21}$) through the Gauss equation,
     \item 
     adding to the right-hand side of the Codazzi equations the viscous terms $\nu \p_{xx}h$, and
     \item 
     passing to the limit $\nu \searrow 0$. 
\end{enumerate}
The regularised terms $h \equiv h^\nu$ serve as the approximate solutions to the Gauss--Codazzi equations, with uniform $L^\infty_\loc$-bounds (\emph{i.e.}, $p=\infty$) established via the method of \emph{invariant regions}. This method can be regarded as a maximum principle for systems of hyperbolic PDEs, which unfortunately requires stringent structural conditions on the source terms, hence severely limiting its applicability. Up to now, isometric immersions of only a few special classes of negatively curved surface metrics have been obtained via this approach \cite{csw', chw, me-ARMA}: they are 1-parameter families of (not necessarily small!) variations of the metrics of classical minimal surfaces, \emph{e.g.}, the catenoid, the helicoid, and the Enneper surface. The Christoffel symbols of these metrics take simple nice forms that satisfy the structural conditions for the method of invariant regions.

The latter papers \cite{chw, me-ARMA} analysed a system of hyperbolic PDEs (to which the method of invariant regions was applied) derived from the Gauss--Codazzi equations. This system is slightly different from that in the pioneering paper \cite{csw'}, and it yields the existence of $C^{1,1}$-isometric immersions of more general families of negatively curved  metrics. Despite all these advantages, one key discoveries in \cite{csw'}, however, went missing from \cite{chw, me-ARMA}: the ``\emph{fluid dynamic formulation}'' of the Gauss--Codazzi equations detailed below.

Considering the ``fluid variables'' $(\rho, u, v, p)$ specified via
\begin{align*}
  \frac{h_{11}}{\sqrt{\det g}} = \rho v^2+p,\qquad  \frac{h_{12}}{\sqrt{\det g}}=-\rho uv,\qquad  \frac{h_{22}}{\sqrt{\det g}}= \rho u^2+p,
\end{align*}
as well as the constitutive relation for the Chaplygin gas:
\begin{align*}
   p = -\rho^{-1}, 
\end{align*}
we see that the Codazzi equations are formally equivalent to the balance law of momentum:
\begin{equation}\label{fluid 1}
    \begin{cases}
    \p_x(\rho u v) + \p_y (\rho v^2+p) = R_1 \equiv -(\rho v^2+p) \G^2_{22} -2\rho uv\G^2_{12} - (\rho u^2+p) \G^2_{11},\\
    \p_x(\rho u^2+p) + \p_y (\rho u v) = R_2 \equiv -(\rho v^2+p) \G^1_{22} -2\rho uv\G^1_{12} - (\rho u^2+p) \G^1_{11},
    \end{cases}
\end{equation}
while the Gauss equation is formally equivalent to the Bernoulli law:
\begin{align}\label{fluid 2}
    \rho = \frac{1}{\sqrt{q^2+K_\M}},\qquad q^2 := {u^2+v^2},
\end{align}
where $K_\M$ is the Gaussian curvature determined by $g$ (Gauss's Theorema Egregium).

Equations~\eqref{fluid 1} and \eqref{fluid 2} constitute the system of compressible Euler equations with source term $(R_1, R_2)^\top$, with $\rho$ interpreted as the fluid density, $(u,v)^\top$ as the flow velocity, and $p$ as the pressure. What makes this analogy more appealing is the following observation in \cite{csw'}: define the sonic speed $c:=\sqrt{p'(\rho)}$ and recall that $q$ is the flow speed; Bernoulli law~\eqref{fluid 2} implies 
\begin{align*}
    c^2-q^2 = K_\M.
\end{align*}
Hence, the cases where $K_\M$ is positive, zero, or negative for the isometric immersions correspond exactly to those where the associated compressible Euler flow is subsonic, sonic, or supersonic.

In summary, the mixed-type scenarios of the Gauss--Codazzi system and its fluid dynamic formulation have such exact correspondence with each other. As commented in \S\ref{sec: concl} below, we hope that further investigations on this correspondence may lead to more results on the elliptic-hyperbolic mixed-type Gauss--Codazzi equations.

\section{Fundamental theorem of submanifold theory with supercritical weak regularity}
\label{sec: Su-Li}

In this section, we sketch a proof of Theorem~\ref{thm: fundamental}, which holds the record of the best regularity assumptions up to date. The key ingredient is an interplay of geometrical (gauge-theoretical) observations and analytical techniques (compensated compactness and Hardy-BMO duality).

\subsection{Historical developments on regularity assumptions}
We first put things into perspectives. The fundamental theorem of submanifold theory --- given a solution $\mathfrak{S}$ to the Gauss--Codazzi--Ricci equations on a simply-connected closed Riemannian manifold $(\M^n,g)$, one can find an isometric immersion $\Phi: (\M^n,g) \to (\R^{n+k},\delta)$ whose extrinsic geometry coincides with $\mathfrak{S}$ --- has also long been a folklore theorem in the $C^\infty$-category. A careful proof was written down by Tenenblat (1971) \cite{ten}.

Over the past decades, lasting endeavours have been devoted to lowering the regularity assumptions. Denote by $\X$ the regularity class of $\Phi$. 
\begin{itemize}
    \item 
    Hartman--Wintner (1950) \cite{hw} established the existence of $\Phi \in \X = C^3$ given $\mathfrak{S}\in C^1$;
\item 
 S. Mardare (2003; 05) \cite{m03, m05} extended the fundamental theorem of surface theory (\emph{i.e.}, $(n,k)=(2,1)$) to $\X = W^{2,\infty}$ given $\mathfrak{S} \in L^\infty$, and then to  $\X = W^{2,p}$ given $\mathfrak{S} \in L^p$ for arbitrary $p>2$;
\item
 Szopos (2008) \cite{sz} obtained the higher dimensional and codimensional analogue; that is, for any $n \geq 2$ one can find an isometric immersion of $(\M^n,g)$ into $\R^{n+k}$ in regularity class $\X = W^{2,p}$,  given $\mathfrak{S} \in L^p$ for arbitrary $p>n$ and $k \geq 1$. Alternative arguments that are geometrical in nature can be found in \cite{cl, me-gauge, m07}.
\end{itemize}
 
S. Mardare \cite{m05} suggested that the fundamental theorem of submanifold theory might be unachievable for $W^{2,p}$-isometric immersions with $p \leq n$ (Remark~\ref{remark: key for W2p}), based on a heuristic argument that treats everything as scalars. Nonetheless, by further exploiting the antisymmetric  structure of the connection 1-form $\Omega$ in the Cartan formalism (compared to Remark~\ref{rem}), X. Su and the author showed in \cite{ls} that the fundamental theorem extends to the \emph{critical} case $\X=W^{2,n}$ for $n \geq 3$; in fact, one may go \emph{slightly supercritical} to $\X = L^{q,n-q}_{2,w}$ for any $q \in ]2,n]$. This is the content of Theorem~\ref{thm: fundamental}. The critical case $(n,k)=(2,1)$ and $\X=W^{2,2}$, on the other hand, was established by Litzinger \cite{litzinger}.

\subsection{Gauge theory}\label{subsec: gauge}
We now change the perspectives: consider the second structural equation~\eqref{second structure eq} as a global identity on the frame bundle over $(\M,g)$.
To construct an isometric immersion $\Phi:(\M,g)\emb \R^{n+k}$, one may solve a \emph{Pfaff} system
\begin{align}\label{pfaff, new}
dP + \Omega P =0 
\end{align}
followed by a \emph{Poincar\'{e}} system 
\begin{equation}\label{poincare, new}
d\Phi = \omega P,
\end{equation}
where $\Omega \in \G\left(T^*\M \otimes \so(n+k)\right)$ is the connection 1-form of the Cartan formalism, $P: \M \to SO(n+k)$ represents a global change of co-ordinates (known as a \emph{gauge transform}), and $\omega:=\big(\omega^1,\ldots,\omega^n,0,\ldots,0\big)^\top$ where $\left\{\omega^i\right\}$ is a local orthonormal coframe on $T^*\M$. The key to the construction of $\Phi$ lies in solving for $P$ satisfying the Pfaff equation~\eqref{pfaff, new}; once $P$ is found, solving the Poincar\'{e} equation~\eqref{poincare, new} is straightforward. The above arguments concerning the Pfaff and Poincar\'{e} systems already appeared in Tenenblat (1971) \cite{ten}, and have been developed and utilised in many later works \cite{m03, m05, m07, cl, cl2, me-gauge, sz, litzinger, new1, new2}.

To find a gauge transform $P$ solving the Pfaff equation~\eqref{pfaff, new}, we recast this problem into the framework of curvatures (of connections on principal fibre bundles). As an affine connection on the frame bundle $\pi: \mathcal{F}\left(T\M \oplus E\right) \to \M$, the connection 1-form $\Omega$ has its \emph{curvature 2-form}:
\begin{equation}\label{def, curvature}
\mathfrak{F}_\Omega := d\Omega+ [\Omega \wedge \Omega] \in \G\left(\M; \bigwedge^2 T^*\M \otimes \sonk\right).
\end{equation}
The second structural equation~\eqref{second structure eq} thus expresses the \emph{flatness} of $\Omega$, namely that $$\mathfrak{F}_\Omega=0.$$

The gauge group (:= the group of gauge transforms) acts on connection 1-forms and curvature 2-forms respectively through
\begin{equation}\label{gauged connection 1-form}
P^\#\Omega := P^{-1}dP + P^{-1}\Omega P
\end{equation}
and 
\begin{equation}\label{gauged curvature}
 P^\#\mathfrak{F}_\Omega = d\left(P^\#\Omega\right) + \left[\left(P^\#\Omega\right)\wedge\left(P^\#\Omega\right)\right].
\end{equation}
Since $P^{-1}$ takes values of invertible matrices \emph{a.e.}, once we obtain $P$ satsifying $P^\#\Omega = 0$, the Pfaff  equation~\eqref{pfaff, new} immediately follows. Therefore, \emph{to prove the existence of an isometric immersion $\Phi: (\M,g)\emb \rnk$ amounts to showing that the connection 1-form $\Omega$ is (locally) gauge-equivalent to the trivial connection}.  To simplify our presentation and only streamline the main ideas of the proof, we omit the ``local-to-global'' issues here. See \cite[\S 3, Proof of Lemma~0.9]{me-gauge} and \cite[\S 4, Step~1]{ls} for details.

Before moving on to the proof of Theorem~\ref{thm: fundamental}, we emphasise that the mathematical theory of principal fibre bundles, which is essentially equivalent to the gauge theory by virtue of T.-T. Wu $\&$ C.-N. Yang's ``dictionary'' (1975) \cite{wy}, has both played a foundational role in modern theoretical physics and brought forth exciting new developments in low-dimensional topology, algebraic geometry, and symplectic geometry, etc. See, \emph{e.g.},  H\'{e}lein \cite{h}.

\subsection{Proof of the supercritical fundamental theorem of submanifold theory}

We now sketch the proof of Theorem~\ref{thm: fundamental} in \cite{ls}. The key analytic tool is a variant of  \cite[Lemma~3.1]{rs}.
\begin{lemma}\label{lem: riviere}
Let $U\subset\R^n$ be a smooth bounded domain, $k$ be a natural number, and $r\in ]2,n[$. There exists some $\e_{\rm Uh} = \e(U,n,k,r)>0$ such that the following holds. Assume that $\Omega \in L^{r,n-r}\left(U;\so(n+k)\otimes\bigwedge^1\R^n\right)$ satisfies $\|\Omega\|_{L^{r, n-r}} \le \e$ for some $\e \in \left]0,\e_{\rm Uh}\right[$. Then one can find $P \in {L_1^{r, n-r}}(U;SO(n+k))$ and   $\xi \in {L_1^{r, n-r}}\left(U;\so(n+k)\otimes\bigwedge^{n-2}\R^n\right) $ such that
\begin{subequations} 
\begin{align}
P^{-1} d P+P^{-1} \Omega P & =\star d \xi & & \text { in } U, \label{xxa, new}\\
d(\star \xi) & =0 & & \text { in } U, \label{xxb, new}\\
\xi & =0  & & \text { on } \partial  U. \label{xxc, new}
\end{align}
\end{subequations} 
Moreover, we have the estimate
\begin{equation} \label{zz, new}
\|d P\|_{L^{r, n-r}}+\|d \xi\|_{L^{r, n-r}} \leq C\|\Omega\|_{L^{r, n-r}} \le C\varepsilon
\end{equation}
for some $C=C(n,k,r)$. 
\end{lemma}

Here $P$ is said to be a \emph{(Coulomb)--Uhlenbeck gauge} of $L^{r,n-r}_1$-regularity. See Uhlenbeck's seminal works \cite{u,u'} and Wehrheim's exposition \cite{w}. Its essential feature is that teh gauge-transformed connection 1-form is coexact; \emph{i.e.}, 
$$d^*\left(P^\#\Omega\right)=0.$$ This is equivalent to Equation~\eqref{xxa, new} by Hodge decomposition and the boundary condition. 

Also note the continuous embedding:
\begin{equation*}
    L^{q,n-q}_{w,\ell} \emb L^{r,n-r}_{\ell}\qquad \text{for any } 2 < r <q \leq n \text{ and } \ell \in \mathbb{N}.
\end{equation*}
In what follows, the index $r$ is understood as taken in this range unless otherwise specified.

\begin{proof}[Sketched proof of Theorem~\ref{thm: fundamental}]

Denote $\Xi := P^\#\Omega\big|_\ball$, where $\ball \subset \R^n$ is a small Euclidean ball. Our goal, as summarised in the penultimate paragraph of \S\ref{subsec: gauge}, is to show that $\Xi=0$. (We focus only on the existence part; the uniqueness part is routine.) In this proof, $C(a_1, a_2, \ldots, a_n)$, $C'(a_1, a_2, \ldots, a_n)$ and the like are generic constants that may change from line to line and that depend only on the parameters $a_1, a_2, \ldots, a_n$. 

By Lemma~\ref{lem: riviere}, Equation~\eqref{xxa, new}, for some $\xi \in {L_1^{r, n-r}}\left(\ball;\so(n+k)\otimes\bigwedge^{n-2}\R^n\right)$ with $2 < r <q \leq n$, one has that $\Xi = \star d \xi$. Thanks to Equation~\eqref{gauged curvature}, $\xi$ satisfies
\begin{equation}\label{xi eq, new}
\begin{cases}
d\star d\xi +\big[\star d\xi \wedge \star d\xi\big] = 0 \qquad \text{in } \ball,\\
\xi = 0 \qquad \qquad \text{on } \p\ball.
\end{cases}
\end{equation}
The boundary condition is taken from \eqref{xxc, new} and, in view of \eqref{xxb, new}, this PDE is elliptic.

Since $L_1^{r, n-r} \emb \dot{W}^{1,2}$ on $\ball$, we may take the wedge product of $\xi$ with Equation~\eqref{xi eq, new} and deduce via integration by parts that
\begin{align*}
 \left\|\star d\xi\right\|^2_{L^2\left(\ball\right)}= \left\| d\xi\right\|^2_{L^2\left(\ball\right)} = -\int_\ball \xi \wedge \star d\xi \wedge \star d\xi.
\end{align*}
Then by $\mathcal{H}^1$ (Hardy space)-BMO duality 
\begin{equation*}
\left\| d\xi\right\|^2_{L^2\left(\ball\right)} \le C(n,k) [\xi]_{BMO(\ball)}  \left\|\big[ \star d\xi \wedge \star d\xi \big]\right\|_{\mathcal{H}^1(\R^n)}.
\end{equation*}
We identify $\xi$ with its extension-by-zero on $\R^n$ that lies in ${L_1^{r, n-r}}\left(\R^n;\so(n+k)\otimes\bigwedge^{n-2}\R^n\right)$. Also, thanks to the endpoint Sobolev embedding, $2< r < q\leq n$, and Lemma~\ref{lem: riviere}, it holds that
\begin{align*}
 [\xi]_{{\rm BMO}(\ball)} \leq C \|d \xi\|_{L^{2, n-2}(\ball)}\leq C\|d \xi\|_{L^{r, n-r}(\ball)}\leq  C\|\Omega\|_{L^{r, n-r}(\ball)},
\end{align*}
where $C=C(n,k,r,q)$. Choosing $\ball$ sufficiently small, we can make $[\xi]_{{\rm BMO}(\ball)} \leq C'(n,k,r)\e$, where $\e \in \left]0,\e_{\rm Uh}\right[$ is an arbitrary constant for $\e_{\rm Uh}$ as in Lemma~\ref{lem: riviere}.  

To bound the Hardy norm $\left\|\big[ \star d\xi \wedge \star d\xi\big] \right\|_{\mathcal{H}^1(\R^n)}=\left\|\big[ \Xi \wedge \Xi \big]\right\|_{\mathcal{H}^1(\R^n)}$, we are once again in the situation of Theorem~\ref{thm: csw} (1) (Chen--Slemrod--Wang \cite{csw}): since $
d\Xi + [\Xi \wedge \Xi]=0$\footnote{
Note that $d\Xi + [\Xi \wedge \Xi]=\mathfrak{F}_{P^\#\Omega}\Big|_\ball$, the localised gauge-transformed curvature 2-form of $\Omega$. By  Equation~\eqref{gauged curvature} it equals to $P^\#\left(\mathfrak{F}_\Omega\right)$, which is zero because $\mathfrak{F}_\Omega=0$ is tantamount to the Gauss--Codazzi--Ricci equations.} and $\Xi \in L^{q,n-q}_{w}(\ball)\emb L^{r, n-r}(\ball) \emb L^r(\ball)$ for any $2 < r <q \leq n$, one may directly adapt the proof of Theorem~\ref{thm: csw} (1) above, together with Remark~\ref{rem+}, to deduce that 
\begin{align*}
\left\|\big[ \star d\xi \wedge \star d\xi\big] \right\|_{\mathcal{H}^1(\R^n)} \leq C(n,r) \|d\xi\|^2_{L^2(\ball)}.
\end{align*}

Therefore, chaining together the previous three estimates, we arrive at 
\begin{align*}
   \|d\xi\|^2_{L^2(\ball)} \leq C''(n,k,r,q) \e   \|d\xi\|^2_{L^2(\ball)}.
\end{align*}
By choosing $\e < 0.99 \left[C''(n,k,r,q)\right]^{-1}$ one concludes that $d\xi=0$ and hence $\Xi = 0$. This is what we want to prove.    \end{proof}

In the proof above, gauge theoretic arguments are crucial for us to go ``supercritical''  on the fundamental theorem of submanifold theory. In turn, the analysis of Gauss--Codazzi--Ricci equations may also shed new lights on the analytical aspects of the gauge theory. In this direction, some tentative exploration has recently been carried out by the author \cite{me-gauge}.

\section{Problems for future investigation}\label{sec: concl}

In this concluding section, we pose several questions for prospective investigations. The list below is by no means exhaustive, and the questions thereof are chosen according only to the  research interest of the author. We do not intend to suggest that these are the most important questions concerning either Gauss--Codazzi--Ricci equations or isometric immersions.

\begin{enumerate}
    \item 
\textbf{Can one further relax the regularity  conditions for the fundamental theorem of submanifold theory?} That is, can we prove the analogue of Theorem~\ref{thm: fundamental} with $\X$ chosen to be a larger function space with coarser topology than $L^{q,n-q}_{2,w}$; $q \in ]2,n]$? 

This question has both theoretical and practical significance. On the one hand, it is closely related to the critical regularity conditions for the solubility of the gauge equation. See Uhlenbeck's seminal works \cite{u, u'} (as well as Wehrheim's exposition \cite{w}) and the important developments \cite{riviere, rs} by Rivi\`{e}re and Struwe. On the other hand, in elasticity theory and numerical computations, the function space of the utmost interest is perhaps $\X=W^{1,\infty}_\loc$ or $W^{1,{\rm BV}}_\loc$; the second derivatives of $\Phi \in \X$ may fail to be well-defined functions. For instance, in finite element methods, \emph{e.g.}, Regge calculus, one considers a  triangulated (hence Lipschitz/PL) manifold equipped with piecewise polynomial $\linf$-metrics, which are non-differentiable in normal directions to the edges. See \u{C}ap--Hu \cite{ch1, ch2} for recent progress on geometric theories of finite element methods in this direction.

We also note that $L^\infty$ and BV-solutions to the Gauss--Codazzi--Ricci equations (hence $W^{2,\infty}$- and $W^{2, {\rm BV}}$-isometric immersions) have been studied by Cao--Huang--Wang \cite{chw2} and Christoforou \cite{chris, chris2} within the framework of conservation laws.

\item 
Concerning the theory of compensated compactness: 
\begin{itemize}
    \item 
\textbf{What is the correct analogue for the wedge product theorem (Robbin--Rogers--Temple \cite{rrt}) with the first-order differential constraints imposed for $d^*$ instead of $d$?} This is particularly relevant to the application of the wedge product theorem to the gauge equation $d\Omega+[\Omega\wedge\Omega]=0$ (and hence to the Gauss--Codazzi--Ricci equations, in view of the Cartan formalism). Indeed, here the first-order differential constraint is the Coulomb--Uhlenbeck gauge condition: $d^*\Omega=0$. 

In the recent preprint \cite{me-wedge}, the author revisited and generalised the wedge product theorem to differential forms with lower regularity assumptions compared to Robbin--Rogers--Temple \cite{rrt}. Our work is largely motivated by Briane--Casado-D\'{i}az--Murat \cite{bcm}. The extension from \cite{bcm} to \cite{me-wedge} is philosophically analogous to the extension from Murat--Tartar's classical div-curl lemma \cite{m, t1, t2} to Robbin--Rogers--Temple \cite{rrt}.

    \item 
\textbf{What is the correct div-curl lemma-type statement for pairings of vector fields (or differential forms) in the Morrey space $L^{2,n-2}$?} In the simplest case, suppose that $E$, $B$ are $L^{2,n-2}_\loc$-vector fields on $\R^n$ such that $\sum_{k=1}^n \p_k E^k=0$ and $\p_i B^j - \p_j B^i=0$ in the sense of distributions for each $1\leq i,j \leq n$. Can one conclude that $E \cdot B \in \mathcal{H}^1_{\rm loc}$?  If the answer is affirmative, then one may extend the fundamental theorem of submanifold theory (Theorem~\ref{thm: fundamental}) to the endpoint case $q=2$ by adapting the arguments in \cite{ls}.  
\end{itemize}


\item 
\textbf{Can we establish global existence theorems of isometric immersions of surface metrics into $\R^3$ with \emph{non-positive} or \emph{sign-changing} Gaussian curvature?}

Arguably the most important aspect of the isometric immersions problem is the existence theory. Nearly all the existence theorems in the literature are about isometric immersions/embeddings of surfaces into $\R^3$ or other 3-manifolds, and are obtained by proving the solubility of Gauss--Codazzi equations. Meanwhile, most of these results assume that the Gaussian curvature $K_\M$ satisfies $K_\M >0$ (\emph{e.g.}, the Weyl problem \cite{w, nir, guanlu, me-BLMS}), $K_\M \geq 0$ \cite{gl}, or $K_\M<0$ \cite{ps, csw', chw, me-ARMA}, namely that the Gauss--Codazzi equations are elliptic, degenerate elliptic, or strictly hyperbolic. Very few existence results are known for the elliptic-hyperbolic mixed-type  or the degenerate hyperbolic  Gauss--Codazzi equations. The only exceptions, to the author's knowledge, are Q. Han's \emph{local} existence theorem for metrics with Gaussian curvature changing sign cleanly \cite{han} (which simplifies the earlier work \cite{cslin}  by C.~S. Lin), and the local existence theorem for isometric immersions of 3-manifold into $\R^6$ by Chen--Clelland--Slemrod--Wang--Yang \cite{ccswy}.

Nash's well-known global existence theorems of $C^1$ and $C^r$-isometric embeddings ($r \geq 3$) \cite{nashC1, nashCk} have no restrictions on the sign of curvature. The curvature is not well-defined for the $C^1$-isometric embeddings constructed using ``Nash wrinkles'' in \cite{nashC1}, while the $C^r$-embeddings in \cite{nashCk} obtained via the Nash--Moser iterations/the ``hard implicit function theorem'' requires high codimensions. The latter case yields a large number of (scalar) Ricci equations in the Gauss--Codazzi--Ricci system, from which little useful geometrical information can be extracted. In brief: Nash's famous results are of a rather different nature from those obtained via the analysis of the Gauss--Codazzi--Ricci equations.

\item 
In continuation of the above problem: \textbf{Can we further explore the fluid dynamic formulation, namely Equations~\eqref{fluid 1} $\&$ \eqref{fluid 2}, to prove global existence theorems for degenerate hyperbolic or mixed-type Gauss--Codazzi equations?}

It would be interesting to see that analytical results and techniques for sonic-supersonic or transonic 2-dimensional Euler flows, which is an extensively studied topic in gas dynamics and quasilinear conservation laws, shed new lights to the existence theory of degenerate hyperbolic or mixed-type Gauss--Codazzi equations. Some tentative explorations were carried out by Slemrod and the author \cite{lslem}, with the primary motivation being establishing a ``renormalised'' fluid dynamic formulation for Nash's $C^1$ (indeed, $C^{1,\alpha}$ for $\alpha < 1/7$) isometric immersion of the flat 2-torus in $\R^3$. 

We may also ask: \textbf{what are the geometrical characterisations (in terms of the Gauss--Codazzi equations and/or isometric immersions) for the shock waves in the weak solutions to the compressible Euler equations?}

In addition, let us bring to the reader's attention \cite{aclsw} by Acharya, G.-Q. Chen, Slemrod, Wang, and the author (2017), which reports some attempts on connecting Gauss--Codazzi equations to more general classes of PDE models in continuum mechanics. Such connections, to the knowledge of the author, remain largely formal up to date.

\item 
The fundamental theorem of surface/submanifold theory (see \cite{ten, sz, m05, m07, ls, cl, cl2, ciarlet', cgm, b, add, add'} among other works) asserts that isometric immersions can be constructed from the solutions to the Gauss--Codazzi--Ricci equations on \emph{simply-connected} domains. Indeed, on simply-connected domains one may solve for the Pfaff and Poincar\'{e} equations to obtain local isometric immersions chart by chart, and then glue them together via a simple patching argument. See  \cite[\S 3, Proof of Lemma~0.9]{me-gauge}; \cite[\S 4, Step~1]{ls} for details.

In view of the discussions above, we pose the question: \textbf{can one formulate and prove a version of the fundamental theorem of submanifold theory on non-simply-connected domains?} The framework of ``branched (isometric) immersions'' due to Gulliver--Osserman--Royden \cite{gor} may be the suitable one for this purpose.

\item 
Last but not least, due to the limitation of space we cannot elaborate on the applications of isometric immersions to the nonlinear elasticity theory, which is a topic of enormous significance. Let us point out, amongst many others,  the following works: 
\begin{itemize}
\item 
\cite{ciarlet-new, cgm, new1, ciarlet'} by P.~G. Ciarlet and coauthors on the intrinsic approach to nonlinear elasticity and the nonlinear Korn's inequality;

    \item 
\cite{elast1, elast2, kms} by Alpern, Kupferman, Maor, and Shachar on the asymptotic rigidity  of elastic shells (\emph{\`{a} la Re\u{s}etnjak} \cite{res1, res2}) and the geometric rigidity of elastic bodies (\emph{\`{a} la Friesecke--James--M\"{u}ller} \cite{fjm}). See also Chen--Slemrod--Li \cite{new2} and Conti--Dolzmann--M\"{u}ller \cite{cdm} for further developments;
\item 
\cite{new-hlp, new-l, new-lmp', new-lmp} by Lewicka, Pakzad, and coauthors on quantitative estimates for (non)-isometric immersibility of metrics, matching of isometries, and the shell theory.
\end{itemize}
We plan to investigate the following question: \textbf{can we further develop the mathematical theory of nonlinear elasticity in any of the above directions; in particular, with the help of recent advances in the analytical theory of Gauss--Codazzi--Ricci equations?}

\end{enumerate}

\medskip

\noindent
{\bf Acknowledgement}. I am deeply indebted to Professor~Gui-Qiang Chen for his patient guidance and unwavering support. Gratitude also goes to my teachers and friends --- Professors~Galia Dafni,  Pengfei Guan, Qing Han, Robert Hardt, Kaibo Hu, Feimin Huang, Jan Kristensen, Fang-Hua Lin, Yuning Liu, Reza Pakzad, Zhongmin Qian, Armin Schikorra, Stephen Semmes, Marshall Slemrod, Dehua Wang, Lihe Wang, Tao Wang, Tian-Yi Wang, and Wei Xiang, among many others --- for generously sharing their insights on topics in this article over all these years. 

I began writing this article during my visit to UCL and Oxford in August 2024. Special thanks to Professors~Hao Ni and Gui-Qiang Chen for their hospitality. The charm of London and Oxford brings me fond and lasting memories that I shall forever cherish.

The research of SL is supported by NSFC Projects 12201399, 12331008, and 12411530065, Young Elite Scientists Sponsorship Program by CAST 2023QNRC001, National Key Research $\&$ Development Program 2023YFA1010900, and the Shanghai Frontier Research Institute for Modern Analysis.

\end{document}